\documentclass{IEEEtran}
\usepackage{cite}
\usepackage{amsmath,amssymb,amsfonts}
\usepackage{algorithmic}
\usepackage{graphicx}
\usepackage{amsthm}
\usepackage{color}
\def\BibTeX{{\rm B\kern-.05em{\sc i\kern-.025em b}\kern-.08em
    T\kern-.1667em\lower.7ex\hbox{E}\kern-.125emX}}

\newtheorem{proposition}{Proposition}[section]
\newtheorem{lemma}[proposition]{Lemma}
\newtheorem{theorem}[proposition]{Theorem}
\newtheorem{corollary}[proposition]{Corollary}

\theoremstyle{definition}
\newtheorem{definition}[proposition]{Definition}

\newcommand{\set}[1]{\left\lbrace #1 \right\rbrace}
\newcommand{\norm}[1]{\left\Vert #1 \right\Vert}
    
\begin{document}
\title{Linear port-Hamiltonian systems\newline are generically controllable\hspace*{25mm}}
\author{Jonas Kirchhoff
\thanks{ }
\thanks{The author is with the Institut für Mathematik, Technische Universität Ilmenau, 98693 Ilmenau, Germany (e-mail: jonas.kirchhoff@tu-ilmenau.de)}}

\maketitle

\begin{abstract}
The new concept of relative generic subsets is introduced. It is shown that the set of controllable linear finite-dimensional port-Hamiltonian systems is a relative generic subset of the set of all linear finite-dimensional port-Hamiltonian systems. This implies that a random, continuously distributed port-Hamiltonian system is almost surely controllable.
\end{abstract}

\begin{IEEEkeywords}
controllability, linear systems, port-Hamiltonian systems
%Enter key words or phrases in alphabetical 
%order, separated by commas. For a list of suggested keywords, send a blank 
%e-mail to keywords@ieee.org or visit \underline
%{http://www.ieee.org/organizations/pubs/ani\_prod/keywrd98.txt}
\end{IEEEkeywords}

\section{Introduction}
\label{sec:introduction}
\noindent
Let~$\mathbb{N}^*$ denote the set of positive integers and~$\mathbb{F}$ be either the field of real or complex numbers.

We consider linear finite-dimensional port-Hamiltonian system in form of an  ordinary differential equation
\begin{align}\label{eq:PHS}
\dot{x} = JHx + Bu
\end{align}
where~$(J,H,B)$ belongs to
\[
\mathcal{PH}_{n,m} := 
\left\{ (J,H,B) \left|
\begin{array}{l}
J=-J^* \in \mathbb{F}^{n\times n}, \ B\in\mathbb{F}^{n\times m},\\
H^* = H \in \mathbb{F}^{n\times n} \ \text{positive definite}
\end{array}
\right.
\hspace*{-2mm}
\right\}.
%  \in\mathcal{PHT}_{n,m}\,\big\vert\,H>0}.
\]
%$J,H\in\mathbb{F}^{n\times n}$ with~$J^* = -J$ and~$H^* = H$,~$B\in\mathbb{F}^{n\times m}$ and locally integrable control~$u\in L^1_{\text{loc}}([0,\infty),\mathbb{R}^m)$ \cite[Definition 2.3.2]{JacobZwart12}.

In the very special case of  ordinary differential systems
\begin{align}\label{eq:ODE}
\dot{x} = Ax + Bu,
\end{align}
%with arbitrary~$(A,B)\in\mathbb{F}^{n\times n}\times \mathbb{F}^{n\times m}$,
 Wonham proved~\cite[Thm.~1.3]{Wonh85} that the set of 
 controllable matrix pairs~$(A,B)\in\mathbb{F}^{n\times n}\times \mathbb{F}^{n\times m}$  is a generic set. 
A generic set is defined  as follows: 
 
 \begin{definition}\label{def:gen}\cite[p.\,28]{Wonh85}
A set~$\mathbb{V}\subseteq\mathbb{F}^N$ is called an \textit{algebraic variety}, if it is the locus of finitely many polynomials in~$N$ indeterminants, i.e.
\begin{align*}
\exists p_1,\ldots,p_{k}\in\mathbb{F}[x_1,\ldots,x_N]: \mathbb{V} = \bigcap_{i = 1}^k p_i^{-1}(\set{0}).
\end{align*}
An algebraic variety~$\mathbb{V}$ is called \textit{proper}, if~$\mathbb{V}\neq\mathbb{F}^N$. We denote the set of algebraic varieties by~$\mathcal{V}_N(\mathbb{F})$ and the set of proper algebraic varieties by~$\mathcal{V}_N^{\text{prop}}(\mathbb{F})$.

A set~$S\subseteq\mathbb{F}^N$ is called \textit{generic}, if there is some proper algebraic variety~$\mathbb{V}$ so that~$S^c\subseteq\mathbb{V}$.
\hfill~$\diamond$
\end{definition}
 
 A first step to generalize generic sets to port-Hamiltonian systems
 could be to consider the  set
 \begin{align*}
\mathcal{F}_{n,m} := \mathbb{F}^{n\times n}\times \mathbb{F}^{n\times n}\times\mathbb{F}^{n\times m}.
\end{align*}
However, 
this set does  not obey the special structure of the
matrices~$(J,H,B) \in \mathcal{PH}_{n,m}$.
Instead, one could  consider~$\mathcal{PH}_{n,m}$ directly,
but~$\mathcal{PH}_{n,m}$ is not a vector space, 
it  is the interior of a convex cone. Thus, in this  case  the concept of genericity 
is not adequate.
To  this  end, I introduce the novel~-- to the best of
my knowledge~-- concept  of \textit{relative generic} sets.
 
 \begin{definition}\label{def:rel_gen}
A set~$S\subseteq\mathbb{F}^N$ is called \textit{relative generic} in~$V\subseteq\mathbb{F}^N$, if
\begin{align*}
\exists\mathbb{V}\in\mathcal{V}_N^{\text{prop}}:\quad V\cap S^c\subseteq\mathbb{V}\quad\wedge\quad V\cap\mathbb{V}^c~\text{is~dense~in}~V.
\end{align*}
\end{definition}
Some comments on this definition are warranted:
Genericity of  a set~$S\subseteq\mathbb{F}^N$
does  not say anything about the properties of~$S$ relative to a set~$V\subseteq\mathbb{F}^N$, except that~${V\cap S^c}\subseteq\mathbb{V}$ for some proper algebraic variety. Since~$V$ and~$\mathbb{V}$ may coincide (or even~$V$ and~$S^c$), this does {not} provide any information about the properties of~$S$ relative to~$V$. Also, if we do not {require} density of~$V\cap\mathbb{V}$ in~$V$, then any set~$S\subseteq\mathbb{F}^N$ is relative generic in any proper algebraic variety~$\mathbb{V}\in\mathcal{V}_N(\mathbb{F})$ and hence we would not gain information for these sets. Note that any proper algebraic variety~$\mathbb{V}$ is a Lebesgue nullset (see\cite[p.\,240]{Fede69}) and hence~$\mathbb{V}^c$ is dense. Thus any generic set~$S$ is relative generic in~$\mathbb{F}^N$ and the concept of relative genericity is a generalization of genericty. 

%\texttt{Diese Motivation sollten Sie nochmal überarbeiten und auch
%die zweite Bedingung~$V\cap\mathbb{V}^c~\text{is~dense~in}~V$ erläutern.
%Ich habe auch  eine weitere Proposition zugefügt. Schließlich führen
%Sie ein neues Konzept ein, dann ist ein wenig `song and dance'
%angebracht. Fallen Ihnen noch weiter Eigenschaften ein,
%die man erwähnen könnte? }

An interesting consequence of  relative genericity is the following.

 \begin{proposition}\label{Prop:dense}
If a  set~$S\subseteq\mathbb{F}^N$ is   relative generic  in~$V\subseteq\mathbb{F}^N$, 
then
$V\cap S^c$ is nowhere dense in~$V$, {i.e. its closure with respect to the relative topology in $V$ has an empty interior}.
The converse is, in general, not true.
\end{proposition}

\begin{proof}
By Definition~\ref{def:rel_gen},~$V\cap S^c\subseteq V\cap \mathbb{V}$ for some proper algebraic variety~$\mathbb{V}.$ Hence the closure of~$V\cap S^c$ with respect to the relative topology in~$V$ is also contained in {closed set}~$V\cap \mathbb{V}$. Since Definition~\ref{def:gen} requires density of~$V\cap\mathbb{V}^c$, we conclude that~$V\cap\mathbb{V}$ does not contain any inner point and is therefore nowhere dense. The fact that any subset of a nowhere dense set is nowhere dense proves that~$V\cap S^c$ is nowhere dense in~$V$.

We give an example to prove that the converse does, in general, not hold. Let~$V = \mathbb{R}_{>0}$ and~$N = 1.$ Let 
\begin{align*}
\varphi: \mathbb{N}^*\to V\cap\mathbb{Q}
\end{align*} 
be {surjective and put
\begin{align*}
S := \bigcup_{i = 1}^{\infty}\left(\varphi(i)-\frac{1}{i^2},\varphi(i)+\frac{1}{i^2}\right)\cap V.
\end{align*}}
Since~$\mathbb{Q}$ is dense,~$S$ is an open and dense set. Hence~$V\cap S^c$ is a nowhere dense subset of~$V$. {The Lebesgue-measure of $S$ is
\begin{align*}
\lambda(S)\leq \sum_{i = 1}^\infty \frac{2}{i^2} = \frac{\pi^2}{3}
\end{align*}
and thus finite. Since any proper algebraic variety is a Lebesgue nullset and $V$ has infinite Lebesgue measure, we find that~$V\cap S^c$ cannot be contained in some proper algebraic variety.}
\end{proof}

This proves that the topological properties of relative generic sets in~$V$ with respect to the relative topology in~$V$ and generic sets with respect to the Euclidean topology in~$\mathbb{F}$ are similiar. This is a mild justification of the concept of relative generic sets.

\section{Main results}
\noindent
We are now  ready to  state the main results of  the present note.

\begin{theorem}\label{Th:PHS_Rel_gen_cont}
The set of controllable port-Hamiltonian systems
\begin{align*}
\mathcal{PH}_{n,m}^{\text{cont}} := \set{(J,H,B)\in\mathcal{PH}_{n,m}\,\big\vert\,~\eqref{eq:PHS}~\text{is~controllable}}
\end{align*}
is relative generic in~$\mathcal{PH}_{n,m}$, where~$n,m\in  \mathbb{N}^*$.
\end{theorem}
The proof is given in  Section~\ref{Sec:Proof}.
\\

As  a consequence of the above result
we find that 
a random, continuously distributed port-Hamiltonian system is almost surely controllable. To this end, define
\begin{align*}
\mathcal{PHT}_{n,m} := \set{(J,H,B)\in\mathcal{F}\,\big\vert\,J^* = -J, H^* = H}
\supset \mathcal{PH}_{n,m} 
\end{align*}
and note that for  any~$(J,H,B)\in\mathcal{PHT}_{n,m}$, 
~$J\in\mathbb{F}^{n\times n}$ is skew-adjoint  and~$H\in\mathbb{F}^{n\times n}$ is self-adjoint or,
 in formal terms,
\begin{align*}
\forall i,j\in\underline{n}: J_{i,j} = -\overline{J_{j,i}}\qquad\text{and}\qquad H_{i,j} = \overline{H_{j,i}},
\end{align*}
where~$\underline{n} := \set{1,\ldots,n}$, and thus~$J$ is uniquely determined by its upper right triangle,~$(J_{i,j})_{(i,j)\in\underline{n}^2,i<j}$, and~$H$ is uniquely determined by the diagonal and its upper right triangle,~$(H_{i,j})_{(i,j)\in\underline{n}^2,i\leq j}.$ This leads to the vector space isomorphism
\begin{align}\label{eq:PHtype}
T \colon \mathcal{PHT}_{n,m}  \to  \mathbb{F}^N,
\end{align}
where
$N := \frac{n(n-1)}{2} + \frac{n(n+1)}{2}+nm$.

%on a continuously distributed random variable~$(A,B)$ be a with values in~$\mathbb{F}^{n\times n}\times\mathbb{F}^{n\times m}$. 
%As mentioned in the  Introduction,
%the set of controllable 
%matrix pairs in~$\mathbb{F}^{n\times n}\times \mathbb{F}^{n\times m}$  is generic. Thus the probability that the random variable~$(A,B)$ describes a controllable ordinary differential system of the form~\eqref{eq:ODE} is~1. Theorem~\ref{Th:PHS_Rel_gen_cont} allows to show similiar result for port-Hamiltonian systems.

We are now ready to formulate  the  corollary.

\begin{corollary}~\label{cor:random}
Let~$T$ denote the vector space isomorphism~\eqref{eq:PHtype}.
Equip~$\mathcal{F}_{n,m}$ with its Borel~$\sigma$-algebra~$\mathfrak{B}$. 
Let~$\mathbb{P}$ be a probability measure on~$\mathcal{PHT}_{n,m}$ 
so that~$\mathbb{P}$ is absolutely continuous with respect to the image measure 
\begin{align*}
\mu: \mathfrak{B}\to\mathbb{R}_{\geq 0}, \quad B\mapsto \lambda(T(B\cap \mathcal{PHT}_{n,m}))
\end{align*}
of the Lebesgue measure~$\lambda$ on~$\mathbb{F}^N$ and~$\mathbb{P}(\mathcal{PH}_{n,m}) = 1.$ Let~$(J,H,B)\sim\mathbb{P}$ be a random variable. Then
\begin{align*}
\mathbb{P}\big((J,H,B)\in\mathcal{PH}_{n,m}^\text{cont}\big) = 1.
\end{align*}
\end{corollary}

\begin{proof}
Since~$\mathcal{PH}_{n,m}$ is a relative open subset of~$\mathcal{PHT}_{n,m}$, a probability measure~$\mathbb{P}$  exists.
Since algebraic varieties are Lebesgue nullsets, the  claim  is a direct consequence of Theorem~\ref{Th:PHS_Rel_gen_cont}.
\end{proof}

%_____________________________________________________________________
\section{Proof of Theorem~\ref{Th:PHS_Rel_gen_cont}}
\label{Sec:Proof}
\noindent
The proof of Theorem~\ref{Th:PHS_Rel_gen_cont}
is based on the following two lemmata on algebraic geometry
properties of relative generic sets.

\begin{lemma}\label{Lem:rel-gen}
Let~$S,V\subseteq\mathbb{F}^N$ be arbitrary sets. If~$V$ is open and~$S\cap V = S'\cap V$ for some generic set~$S'$, then~$S$ is relative generic in~$V$.
\end{lemma}

%\texttt{Wo kommt Lemma~\ref{Lem:rel-gen}  zum Einsatz?}

\begin{proof}
If~$S\cap V = S'\cap V$, then by genericity of~$S'$ there exists~$\mathbb{V}\in\mathcal{V}_N^{\text{prop}}(\mathbb{F})$ such that
\begin{align*}
V\cap S^c = V\cap (S')^c\subseteq\mathbb{V}\in\mathcal{V}_N^{\text{prop}}(\mathbb{F}).
\end{align*} 
It remains to prove that~$V\cap \mathbb{V}^c$ is dense in~$V$. Seeking a contradiction suppose there exists some~$\hat{x}\in V$ and~$\varepsilon >0$ so that
\begin{align*}
\set{x\in V\,\big\vert\,\norm{x-\hat{x}}_2<\varepsilon}\cap\mathbb{V}^c = \emptyset.
\end{align*}
Since~$V$ is open, we can choose without loss of generality~$\varepsilon>0$ such that 
\begin{align*}
\set{x\in  V\,\big\vert\,\norm{x-\hat{x}}_2<\varepsilon} 
= \set{x\in\mathbb{F}^N\,\big\vert\,\norm{x-\hat{x}}_2<\varepsilon}.
\end{align*}
%\texttt{(Ich habe  oben~$x\in \mathbb{V}$  zu~$x\in V$  gemacht;
%und jetzt folgend
%`Hence~$ x$ is an inner point'
%zu``Hence~$\hat x$ is an inner point''. Korrekt?)}
Hence~$\hat x$ is an inner point of~$\mathbb{V}$, which contradicts the fact that any proper algebraic variety is a closed Lebesgue nullset and hence nowhere dense (see\cite[p.\,240]{Fede69}). This proves the lemma.
\end{proof}

%\texttt{Im folgenden Lemma und im Beweis haben Sie  an diversen Stellen~$\mathbb{R}$
%geschrieben, das soll~$\mathbb{F}$ sein, oder?}

\begin{lemma}~\label{lem:simplification}
Let~$U\subseteq\mathbb{F}^N$ be a subspace,
$S,V\subseteq U$,
and 
$T:U\to\mathbb{F}^M$  a vector space isomorphism
for~$M=\dim U$.
 Then~$S$ is relative generic in~$V$ if, and only if,~$TU$ is relative generic in~$TV$. 
\end{lemma}

\begin{proof}
If~$V = \emptyset$, then any set~$S$ is relative generic in~$V$ and nothing has to be shown. If~$S = \emptyset$, then~$S$ is relative generic in~$V$ if, and only if,~$V = \emptyset$ and we are in the first case. Hence let~$S\neq\emptyset$ and~$V\neq\emptyset$.

\noindent
$\implies$ Let~$S$ be relative generic in~$V$. Then there exists some proper algebraic variety~$\mathbb{V}\subseteq\mathbb{F}^N$ so that~$V\cap S^c\subseteq\mathbb{V}$ and~$V\cap \mathbb{V}^c$ is dense in~$V$. Since isomorphisms on finite dimensional vector spaces are {continuous} this implies {that}~$TV\cap (TS)^c\subseteq T(\mathbb{V}\cap U)$ and~$TV\cap T(\mathbb{V}^c\cap U)$ is dense in~$TV$. Since the intersection of proper algebraic varieties is a proper algebraic variety and proper subspaces are proper algebraic varieties, too, we conclude that~$\mathbb{V}\cap U \in\mathcal{V}_N^{\text{prop}}(\mathbb{F})$. Further we find polynomials~$p_1,\ldots,p_k\in\mathbb{F}[x_1,\ldots,x_N]$ such that~$U\cap\mathbb{V} = \bigcap_{i = 1}^kp_i^{-1}(\set{0})$ and thus
\begin{align*}
T(U\cap\mathbb{V}) = \set{x\in\mathbb{F}^M\,\big\vert\,
\forall i =1,\ldots,k\
%\underline{k}
: p_i(T^{-1}x) = 0}.
\end{align*}
%where~$\underline{k} := \set{1,\ldots,k}$.
Since~$T^{-1}$ can be represented as a matrix in~$\mathbb{F}^{N\times M}$, it follows that~$p_i\circ T^{-1}$ is a polynomial in~$M$ indeterminants. This proves that~$T(U\cap\mathbb{V})$ is an algebraic variety. Since~$TV$ is nonempty and~$TV\cap T(\mathbb{V}^c\cap U)$ is dense in~$TV$,~$T(U\cap V)$ is a proper algebraic variety.

\noindent
$\impliedby$ Let~$TU$ be relative generic in~$TV$. Then there exists some~$\mathbb{V}\in\mathcal{V}_M^{\text{prop}}(\mathbb{F})$ so that the inclusion~$TV\cap TU\subseteq\mathbb{V}$ holds true and~$TV\cap \mathbb{V}^c$ is dense in~$TV$. From standard linear algebra we know that there is a vector space endomorphism~$T':\mathbb{F}^N\to\mathbb{F}^N$ so that~$T'U = \mathbb{F}^M\times\set{0}^{N-M}$ and 
\begin{align*}
\forall u\in U: T'u = (Tu,0_{N-M}).
\end{align*}
By Definition~\ref{def:gen} we find~$p_1,\ldots,p_k\in\mathbb{F}[x_1,\ldots,x_M]$ so that~$\mathbb{V} = \bigcap_{i = 1}^kp_i^{-1}(\set{0}).$~$\mathbb{F}[x_1,\ldots,x_M]$ is embedded into~$\mathbb{F}[x_1,\ldots,x_N]$
 and thus we find
\begin{align*}
(T')^{-1}\mathbb{V}\times\mathbb{F}^{N-M} = \set{x\in \mathbb{F}^N\,\big\vert\,\forall i =1,\ldots,k\ : p_i(T'x) = 0}.
\end{align*}
Since~$T'$ can be represented by a matrix in~$\mathbb{F}^{N\times N},$ it is follows that~$(T')^{-1}\mathbb{V}\times\mathbb{F}^{N-M}\in\mathcal{V}_N(\mathbb{F})$.~$(T')^{-1}\mathbb{V}$ is proper since~$\mathbb{V}$ is proper. Hence
\[
T^{-1}\mathbb{V} 
= \left( (T')^{-1}\mathbb{V}\times\mathbb{F}^{N-M}\right) \cap U
\] 
is also a proper algebraic variety. Since~$T^{-1}$ is a diffeomorphism we find that~$T^{-1}\mathbb{V}^c\cap V$ is dense in~$V$. Finally, the inclusion~$$S^c\cap V\subseteq T^{-1}\mathbb{V}$$ proves the lemma.
\end{proof}

\noindent
\textbf{Proof of Theorem~\ref{Th:PHS_Rel_gen_cont}:}
We proceed in  two steps.
\\
\noindent{\sc Step 1}: \
 We show that the set
\begin{align*}
S := \set{(J,H,B)\in\mathcal{PHT}_{n,m}\,\big\vert\,~\eqref{eq:PHS}~\text{is~controllable}}
\end{align*}
as a subset of~$\mathbb{F}^N$ is the complement of a proper algebraic variety~{$\mathbb{V}\in\mathcal{V}^{\text{prop}}_{N}(\mathbb{F})$}. Write, for~$(J,H,B)\in\mathcal{PHT}_{n,m}$, the Kalman matrix of~$(JH,B)$ as 
\begin{align*}
\mathcal{K}(J,H,B) := [B,JHB,\ldots,(JH)^{n-1}B]\in\mathbb{F}^{n\times nm}.
\end{align*}
Let further be~$\widetilde{M}_1,\ldots,\widetilde{M}_q$ be all minors of order~$n$ {with respect to}~$\mathbb{F}^{n\times nm}$ and put, for all~$i\in\set{1,\ldots,q}$,
\begin{align*}
M_i: \mathcal{PHT}_{n,m}\to\mathbb{F}, 
\quad
(J,H,B)\mapsto\widetilde{M}_i(\mathcal{K}(J,H,B)).
\end{align*}
It is easy to see (e.g.\ by the Leibniz formula) that~$M_1,\ldots,M_q$ {are polynomials in~$N$ indeterminants}. The Kalman rank criterion~\cite[Theorem 3, p.\,89]{Sont98a} yields
\begin{align*}
S^c = \bigcap_{i = 1}^q M_i^{-1}(\set{0}) =: \mathbb{V}\in\mathcal{V}_N(\mathbb{F}).
\end{align*}
Hence it remains to prove that~$\mathbb{V}$ is proper,
i.e.~~$S\neq\emptyset$. For~\mbox{$n = 1$} this is trivial. Let~$n\geq 2$ and choose
\begin{align*}
J = 
\left[\begin{smallmatrix}
0 & -1 \\
1 & \ddots & \ddots\\
 & \ddots & \ddots & -1\\
 & & 1 & 0
\end{smallmatrix}
\right],\quad H = I_n,\quad B = [e_1,0_{n\times(m-1)}],
\end{align*}
where~$e_1$ is the first standard unit vector in~$\mathbb{F}^n$.
Then a simple calculation gives  
\begin{align*}
\text{im}\,\mathcal{K}(J,H,B) = \text{im}\,[e_1,Je_1,\ldots,J^{n-1}e_1] = \mathbb{F}^n
\end{align*}
and hence~$(J,H,B)\in S$.
%
%
%Then we find for the first standard unit vector~$e_1\in\mathbb{F}^n$
%\begin{align*}
%\forall k\in\set{0,\ldots,n-1}: J^ke_1\in\text{span}\set{e_1,\ldots,e_{k+1}}.
%\end{align*}
%With~$B := [e_1,0_{n\times (m-1)}]$ we find that the Kalman matrix of~$JH = J$ and~$B$ has rank~$n$ and thus~$(J,H,B)\in S$.
This proves that~$S$ is the complement of a proper algebraic variety.

\noindent{\sc Step 2}: \
Note that~$\mathcal{PH}_{n,m}$ is open in~$\mathcal{PHT}_{n,m}$ since the spectrum of a quadratic matrix is continuous {with respect to the Hausdorff-distance~(see~\cite[Theorem 3]{BEK90})}. Hence the equality~$\mathcal{PH}_{n,m}^{\text{cont}} = \mathcal{PH}_{n,m}\cap S$ and Lemma~\ref{Lem:rel-gen} imply that~$\mathcal{PH}_{n,m}^{\text{cont}}\subseteq\mathbb{F}^N$ is relative generic in~$\mathcal{PH}_{n,m}\subseteq\mathbb{F}^N.$ Applying Lemma~\ref{lem:simplification} {to an isomorphism from $\mathcal{PHT}_{n,m}$ to $\mathbb{F}^{N}$} we find that~$\mathcal{PH}_{n,m}^{\text{cont}}\subseteq\mathcal{F}_{n,m}$ is relative generic in~$\mathcal{PH}_{n,m}\subseteq\mathcal{F}_{n,m}.$ This completes the proof of the theorem.
\hfill~$\Box$

\section*{Acknowledgment}
\noindent
I thank Professor Achim Ilchmann (Ilmenau) for his valuable advice. Further I thank my fellow student Tobias Bernstein (Ilmenau) for his proofreading.

\end{document}